\documentclass[10pt]{amsart}

\usepackage{graphicx,mathrsfs,tikz,fancyhdr}
\usepackage{amsmath,amsthm, amssymb}
\usepackage{thmtools}
\usepackage{nameref}
\usepackage[backend=bibtex]{biblatex}
\usepackage{hyperref}
\usepackage[capitalize]{cleveref}
\hypersetup{colorlinks=true, urlcolor=blue, citecolor=blue, linkcolor=blue}
\bibliography{Masseybib.bib}

\DeclareMathOperator{\mult}{mult}

\newcommand{\no}{\backslash}
\newcommand{\pd}{\partial}

\newcommand{\newword}[1]{\textbf{\emph{#1}}}
\newcommand{\wt}{\widetilde}

\newcommand{\CC}{\mathbb{C}}

\newcommand{\ZZ}{\mathbb{Z}}

\newcommand{\cO}{\mathcal{O}}

\newcommand{\cU}{\mathcal{U}}

\newcommand{\cW}{\mathcal{W}}

\newcommand{\bb}[1]{\textbf{#1}}

\newcommand{\lm}{\lambda}

\newcommand{\0}{\mathbf{0}}

\numberwithin{equation}{section}

\newtheorem{defn0}{Definition}[section]
\newtheorem{prop0}[defn0]{Proposition}
\newtheorem{conj0}[defn0]{Conjecture}
\newtheorem{thm0}[defn0]{Theorem}
\newtheorem{lem0}[defn0]{Lemma}
\newtheorem{corollary0}[defn0]{Corollary}
\newtheorem{example0}[defn0]{Example}
\newtheorem{remark0}[defn0]{Remark}
\newtheorem{question0}[defn0]{Question}

\newenvironment{defn}{\begin{defn0}}{\end{defn0}}
\newenvironment{prop}{\begin{prop0}}{\end{prop0}}
\newenvironment{conj}{\begin{conj0}}{\end{conj0}}
\newenvironment{thm}{\begin{thm0}}{\end{thm0}}

\newenvironment{cor}{\begin{corollary0}}{\end{corollary0}}
\newenvironment{exm}{\begin{example0}\rm}{\end{example0}}
\newenvironment{rem}{\begin{remark0}\rm}{\end{remark0}}

\crefformat{equation}{Eq. ~(#2#1#3)}
\crefname{theorem}{Theorem}{Theorems}
\crefname{proposition}{Proposition}{Propositions}
\crefname{corollary}{Corollary}{Corollaries}
\crefname{lemma}{Lemma}{Lemmas}
\crefname{section}{section}{sections}
\crefname{subsection}{subsection}{subsections}
\crefname{example}{Example}{Examples}
\crefname{conjecture}{Conjecture}{Conjectures}

\makeatletter
\newcommand{\myref}[1]{%
  \begingroup%
    \edef\@mytxt{\csname r@#1\endcsname}%
    \edef\@mytst{\expandafter\@thirdoffive\@mytxt}%
    \ifx\@mytst\empty\relax%
      \cref{#1}%
    \else%
      \nameref{#1}~\labelcref{#1}%
    \fi%
  \endgroup}
\newcommand{\Myref}[1]{%
  \begingroup%
    \edef\@mytxt{\csname r@#1\endcsname}%
    \edef\@mytst{\expandafter\@thirdoffive\@mytxt}%
    \ifx\@mytst\empty\relax%
      \Cref{#1}%
    \else%
      \nameref{#1}~\labelcref{#1}%
    \fi%
  \endgroup}
\makeatother

\author{Brian Hepler}
\address{Department of Mathematics\\
  Northeastern University\\
  Boston, Massachusetts 02115}
\email[B.~Hepler]{hepler.b@husky.neu.edu}
\author{David B. Massey}
\address{Department of Mathematics\\
  Northeastern University\\
  Boston, Massachusetts 02115}
\email[D.~Massey]{d.massey@neu.edu}

\title{Some Special Cases of Bobadilla's Conjecture}

\keywords{L\^e's conjecture, Bobadilla's conjecture, Milnor fiber, $1$-dimensional critical locus, hypersurface}

\subjclass[2010]{32B15, 32C18, 32B10, 32S25, 32S15, 32S55}

\date{}

\begin{document}

\begin{abstract}
We prove two special cases of a conjecture of J. Fern\'andez de Bobadilla for hypersurfaces with $1$-dimensional critical loci.  

We do this via a new numerical invariant for such hypersurfaces, called the beta invariant, first defined and explored by the second author in 2014.  The beta invariant is an algebraically calculable invariant of the local ambient topological-type of the hypersurface, and the vanishing of the beta invariant is equivalent to the hypotheses of Bobadilla's conjecture. 
\end{abstract}

\maketitle

\thispagestyle{fancy}

\lhead{}
\chead{}
\rhead{ }

\lfoot{}
\cfoot{}
\rfoot{}

\section{Introduction}

Throughout this paper, we shall suppose that $\cU$ is an open neighborhood of the origin in $\CC^{n+1}$, and that $f: (\cU,\bb{0}) \to (\CC,0)$ is a complex analytic function with a 1-dimensional critical locus at the origin, i.e., $\dim_\bb{0} \Sigma f = 1$.  We use coordinates $\bb z:=(z_0,\cdots,z_n)$ on $\cU$.  

We assume that $z_0$  is generic enough so that  $\dim_\bb{0} \Sigma (f_{|_{V(z_0)}} ) = 0$.  One implication of this is that
$$
V\left(\frac{\partial f}{\partial z_1}, \frac{\partial f}{\partial z_2}, \dots, \frac{\partial f}{\partial z_n}\right)
$$
is purely 1-dimensional at the origin. As analytic cycles, we write
$$
\left[V\left(\frac{\partial f}{\partial z_1}, \frac{\partial f}{\partial z_2}, \dots, \frac{\partial f}{\partial z_n}\right)\right]  \ = \ \Gamma^1_{f, z_0} + \Lambda^1_{f, z_0},
$$
where $\Gamma^1_{f, z_0}$ and  $\Lambda^1_{f, z_0}$ are, respectively, the relative polar curve and 1-dimensional L\^e cycle; see \cite{lecycles} or the section.

\medskip

We recall a classical non-splitting result (presented in a convenient form here) proved independently by Gabrielov, Lazzeri, and L\^e (in \cite{gabrielov}, \cite{lazzerimono}, and \cite{leacampo}, respectively) regarding the non-splitting of the cohomology of the Milnor fiber of $f_{|_{V(z_0)}}$ over the critical points of $f$ in a nearby hyperplane slice $V(z_0 -t)$ for a small non-zero value of $t$.

\begin{thm}[GLL non-splitting]\label[theorem]{nosplit} The following are equivalent:
\begin{enumerate}

\item The Milnor number of $f_{|_{V(z_0)}}$ at the origin is equal to
\begin{align*}
\sum_C \mu^\circ_{{}_C} \left (C \cdot V(z_0) \right )_\0,
\end{align*}
where the sum is over the irreducible components $C$ of $\Sigma f$ at $\0$, $\left (C \cdot V(z_0) \right)_\0$ denotes the intersection number of $C$ and $V(z_0)$ at $\0$, and $\mu^\circ_{{}_C}$ denotes the Milnor number of $f$, restricted to a generic hyperplane slice, at a point $\bb{p} \in C \no \{\0 \}$ close to $\0$.  

\medskip

\item $\Gamma^1_{f, z_0}$ is zero at the origin (i.e., $\0$ is not in the relative polar curve).

\end{enumerate}
\medskip
Furthermore, when these equivalent conditions hold, $\Sigma f$ has a single irreducible component which is smooth and is transversely intersected by $V(z_0)$ at the origin. 
\end{thm}

This paper is concerned with a recent conjecture made by Javier Fern\'andez de Bobadilla, positing that, in the spirit of \cref{nosplit}, the cohomology of the Milnor fiber of $f$, {\bf not of a hyperplane slice}, does not split.  We state a slightly more general form of Bobadilla's original conjecture, for the case where $\Sigma f$ may, a priori, have more than a single irreducible component:

\begin{conj}[Fern\'andez de Bobadilla]\label[conjecture]{bobagen}
Denote by $F_{f,\0}$ the Milnor fiber of $f$ at the origin. Suppose that $\wt H^*(F_{f,\0};\ZZ)$ is non-zero only in degree $(n-1)$, and that 
\begin{align*}
\wt H^{n-1}(F_{f,\0};\ZZ) \cong \bigoplus_C \ZZ^{\mu^\circ_{{}_C}}
\end{align*}
where the sum is over all irreducible components $C$ of $\Sigma f$ at $\0$. Then, in fact, $\Sigma f$ has a single irreducible component, which is smooth. 

\end{conj}

Bobadilla's conjecture, in its original phrasing (\cite{bobleconj}), is a reformulation of a conjecture of L\^e (see, for example, \cite{leconj}): if $(X,\0)$ is a reduced surface germ in $(\CC^3,\0)$, and the (real) link of $X$ is homeomorphic to a sphere, then $X$ is (analytically) isomorphic to the total space of an equisingular deformation of an irreducible plane curve.  

\smallskip

We approach \cref{bobagen} via the \bb{beta invariant} of a hypersurface with a $1$-dimensional critical locus, first defined and explored by the second author in \cite{betainv}.  The beta invariant, $\beta_f$, of $f$ is an invariant of the local ambient topological-type of the hypersurface $V(f)$. It is a non-negative integer, and is algebraically calculable. 

\smallskip
Our motivation for using this invariant is that the requirement that $\beta_f = 0$ is precisely equivalent to the hypotheses of \cref{bobagen}, essentially turning the problem into a purely algebraic question \cite[see][Theorem 5.4]{betainv} For this reason, we will refer to our new formulation of  \labelcref{bobagen} as the \bb{Beta Conjecture}.  

\smallskip

\bigskip

In this paper, we give proofs of the \nameref{beta0} in two special cases: 

\begin{enumerate}

\item In \cref{betainduction}, we prove an induction-like result for when $f$ is a sum of two analytic functions defined on disjoint sets of variables.  

\smallskip

\item In \cref{onlythm}, we prove the result for the case when the relative polar curve $\Gamma_{f,z_0}^1$ is defined by a single equation inside the relative polar  surface $\Gamma_{f,\bb{z}}^2$ (see below).

\end{enumerate}

\bigskip

\section{Notation and Known Results}
The bulk of this section is largely a summary of the concepts of Chapter 1 of \cite{lecycles}, which will be used throughout this paper.  

\bigskip

Our assumption that $\dim_\bb{0} \Sigma (f_{|_{V(z_0)}}) = 0$  is equivalent to assuming that the variety $V \left (\frac{\pd f}{\pd z_1},\cdots,\frac{\pd f}{\pd z_n} \right )$ is purely 1-dimensional (and non-empty) at $\0$ and is intersected properly by the hyperplane $V(z_0)$ at $\bb{0}$.

\begin{defn}\label[definition]{cycles}

The \newword{relative polar surface of $f$ with respect to $\bb z$}, denoted $\Gamma_{f,\bb z}^2$\,, is, as an analytic cycle at the origin, $\left [V \left (\frac{\pd f}{\pd z_2},\cdots,\frac{\pd f}{\pd z_n} \right ) \right ]$. Note that each component of this at the origin must be precisely $2$-dimensional, and so is certainly not contained in $\Sigma f$.

\medskip

The \newword{relative polar curve of $f$ with respect to $z_0$}, denoted $\Gamma_{f,z_0}^1$, is, as an analytic cycle at the origin, the collection of those components of the cycle $\left [V \left (\frac{\pd f}{\pd z_1},\cdots,\frac{\pd f}{\pd z_n} \right ) \right ]$ which are not contained in $\Sigma f$.

\medskip

The \newword{1-dimensional L\^e cycle of $f$ with respect to $z_0$}, at the origin, denoted $\Lambda_{f,z_0}^1$, consists of those components of $\left [V \left (\frac{\pd f}{\pd z_1},\cdots,\frac{\pd f}{\pd z_n} \right )\right ]$ at the origin which \emph{are} contained in $\Sigma f$.  
\end{defn}

\medskip

We sometimes enclose an analytic variety $V$ in brackets to indicate that we are considering $V$ as a cycle.  We do, however, frequently omit this notation if it is clear from context that a given variety is to be considered as an analytic cycle. 

\smallskip

An immediate consequence of \cref{cycles} is that, as cycles on $\cU$, 
\begin{align*}
V \left (\frac{\pd f}{\pd z_1},\cdots,\frac{\pd f}{\pd z_n} \right ) = \Gamma_{f,z_0}^1 + \Lambda_{f,z_0}^1 .
\end{align*}
We will use this identity throughout this paper.  

\smallskip

Note that, by assumption, $V \left ( \frac{\pd f}{\pd z_0} \right )$ properly intersects $\Gamma_{f,z_0}^1$ at $\bb{0}$, and also that $V(z_0)$ properly intersects $\Lambda_{f,z_0}^1$ at $\bb{0}$. 

Letting $C$'s denote the underlying reduced components of $\Sigma f$ at $\bb{0}$, we have (as cycles at the origin)
\begin{align*}
\Lambda_{f,z_0}^1 &= \sum_C \mu^\circ_{{}_C} [C],
\end{align*}
where $\mu^\circ_{{}_C}$  denotes the Milnor number of $f$, restricted to a generic hyperplane slice, at a point $\bb{p} \in C \no \{\0 \}$ close to $\0$ (\cite[see][Remark 1.19]{lecycles}).

\begin{defn}\label[definition]{lenums}
The intersection numbers $\left ( \Gamma_{f,z_0}^1 \cdot V \left (\frac{\pd f}{\pd z_0} \right ) \right )_\0$ and $\left (\Lambda_{f,z_0}^1 \cdot V (z_0) \right )_\0$ are, respectively, the  \newword{L\^{e} numbers} $\lm_{f,z_0}^0$ and $\lm_{f,z_0}^1$ (at the origin).  

Via the above formula for $\Lambda_{f,z_0}^1$, we have:
\begin{align*}
\lm_{f,z_0}^1 &= \sum_C \mu^\circ_{{}_C} \left (C \cdot V(z_0) \right )_\0.
\end{align*}
\end{defn}

A fundamental property of L\^e numbers from \cite{lecycles} is:

\begin{prop}\label[proposition]{BettiLe}
 Let $\wt b_n(F_{f,\0})$ and  $\wt b_{n-1}(F_{f,\0})$ denote the reduced Betti numbers of the Milnor fiber of $f$ at the origin. Then,
$$
\wt b_n(F_{f,\0})-\wt b_{n-1}(F_{f,\0}) \ = \ \lm_{f,z_0}^0-\lm_{f,z_0}^1.
$$
\end{prop}

\bigskip

We will need the following classical relations between intersection numbers. 

\begin{prop}\label[proposition]{intnums}
Since $\dim_\0 \Sigma (f_{|_{V(z_0)}}) = 0$\textnormal{:}
\begin{enumerate} 
\item $\dim_\0 \Gamma_{f,z_0}^1 \cap V(f) \leq 0$, $\dim_\0 \Gamma_{f,z_0}^1 \cap V(z_0) \leq 0$, $\dim_\0 \Gamma_{f,z_0}^1 \cap V \left (\frac{\pd f}{\pd z_0} \right ) \leq 0$, and
\begin{align*}
\left ( \Gamma_{f,z_0}^1 \cdot V(f) \right )_\0 = \left ( \Gamma_{f,z_0}^1 \cdot V(z_0) \right )_\0 + \left (\Gamma_{f,z_0}^1 \cdot V \left ( \frac{\pd f}{\pd z_0} \right)  \right )_\0.
\end{align*}
The proof of this result is sometimes referred to as \emph{Teissier's trick}.

\bigskip

\item In addition,
\begin{align*}
\mu_\0 \left ( f_{|_{V(z_0)}} \right ) = \left ( \Gamma_{f,z_0}^1 \cdot V(z_0) \right )_\0 + \left ( \Lambda_{f,z_0}^1 \cdot V(z_0) \right )_\0.
\end{align*}

\end{enumerate}

\end{prop}

\medskip

Formula $(1)$ above was first proved by B. Teissier in \cite{teissiercargese} for functions with isolated critical points, and it is an easy exercise to show that the result still holds in the case where $f$ has a critical locus of arbitrary dimension.  Formula $(2)$ follows from the fact that
\begin{align*}
\Sigma \left ( f|_{V(z_0)} \right ) = V \left ( z_0,\frac{\pd f}{\pd z_1}, \cdots,\frac{ \pd f}{\pd z_n} \right )
\end{align*}
\medskip 
and the fact that $V(z_0)$ properly intersects $V \left ( \frac{\pd f}{\pd z_1},\cdots,\frac{\pd f}{\pd z_n} \right )$ at the origin.  

\smallskip

The following numerical invariant, defined and discussed in \cite{betainv}, is crucial to the contents and goal of this paper.

\begin{defn}\label[definition]{beta}
The \newword{beta invariant} of $f$ with respect to $z_0$ is:
\begin{align*}
\beta_f = \beta_{f,z_0} &:= \left ( \Gamma_{f,z_0}^1 \cdot V \left (\frac{\pd f}{\pd z_0} \right ) \right )_\bb{0} -  \sum_C \mu^\circ_{{}_C} \left [ \left ( C \cdot V(z_0) \right )_\bb{0} -1 \right ] \\
&= \lm_{f,z_0}^0 - \lm_{f,z_0}^1 + \sum_C \mu^\circ_{{}_C}\\
&= \wt b_n(F_{f,\0})-\wt b_{n-1}(F_{f,\0}) + \sum_C \mu^\circ_{{}_C}.
\end{align*}
\end{defn}
Using \cref{intnums}, $\beta_f$ may be equivalently expressed as
\begin{align*}
\beta_f &= \left ( \Gamma_{f,z_0}^1 \cdot V(f) \right )_\0 - \mu_\0 \left ( f_{|_{V(z_0)}} \right ) + \sum_C \mu^\circ_{{}_C}.
\end{align*}

\medskip

\begin{rem}
A key property of the beta invariant is that the value $\beta_f$ is independent of the choice of  linear form $z_0$ (provided, of course, that the linear form satisfies $\dim_\0 \Sigma (f_{|_{V(z_0)}}) =0$).  This often allows a great deal of freedom in calculating $\beta_f$ for a given $f$, as different choices of linear forms $L = z_0$ may result in simpler expressions for the intersection numbers $\lm_{f,z_0}^0$ and $\lm_{f,z_0}^1$, while leaving the value of $\beta_f$ unchanged.  \cite[See][Remark 3.2, Example 3.4]{betainv}.
\end{rem}

\medskip

It is shown in  \cite{betainv} that $\beta_f\geq 0$. The interesting question is how strong the requirement that $\beta_f=0$ is.
\medskip

\begin{conj}[Beta Conjecture]\label[conjecture]{beta0}
If $\beta_f=0$,  then $\Sigma f$ has a single irreducible component at $\0$, which is smooth.  
\end{conj}

\smallskip 

\begin{conj}[polar form of the Beta Conjecture]\label[conjecture]{beta0polar}
If $\beta_f=0$, then $\0$ is not in the relative polar curve $\Gamma_{f,z_0}^1$ (i.e., the relative polar curve is $0$ as a cycle at the origin). 

\medskip

Equivalently, if the relative polar curve at the origin is not empty, then $\beta_f>0$.
\end{conj}

\smallskip

\begin{prop}\label[proposition]{betaEquiv}
The \nameref{beta0} is equivalent to the \nameref{beta0polar}.
\end{prop}

\begin{proof} Suppose throughout that $\beta_f=0$.

\smallskip

Suppose first that the Beta Conjecture holds, so that $\Sigma f$ has a single irreducible component at $\0$, which is smooth. Then $\beta_f= \lm_{f,z_0}^0=0$, and so the relative polar curve must be zero at the origin.

\smallskip

Suppose now that the polar form of the Beta Conjecture holds, so that $\Gamma_{f,z_0}^1 = 0$ at $\0$. Then \nameref{nosplit}  implies that $\Sigma f$ has a single irreducible component at $\0$, which is smooth.
\end{proof}

\medskip

\bigskip

\section{Generalized Suspension}

Suppose that $\cU$ and $\cW$ are open neighborhoods of the origin in $\CC^{n+1}$ and $\CC^{m+1}$, respectively, and let $g: (\cU,\bb{0}) \to (\CC,0)$ and $h:(\cW,\bb{0}) \to (\CC,0)$ be two complex analytic functions.  Let $\pi_1 : \cU \times \cW \to \cU$ and $\pi_2 : \cU \times \cW \to \cW$ be the natural projection maps, and set $f=g\boxplus h := g \circ \pi_1 + h \circ \pi_2$.  Then, one trivially has
\begin{align*}
\Sigma f &= \left ( \Sigma g \times \CC^{m+1} \right ) \cap \left ( \CC^{n+1} \times \Sigma h \right ).
\end{align*}
Consequently, if we assume that $g$ has a one-dimensional critical locus at the origin, and that $h$ has an isolated critical point at $\bb{0}$, then $\Sigma f = \Sigma g \times \{\bb{0}\}$ is 1-dimensional and (analytically) isomorphic to $\Sigma g$. 

\smallskip

From this, one immediately has the following result.

\begin{prop0}\label[proposition]{gensusp}
Suppose that $g$ and $h$ are as above, so that $f = g\boxplus h$ has a one-dimensional critical locus at the origin in $\CC^{n+m+2}$. Then, $\beta_f = \mu_\0(h)\beta_g$. 
\end{prop0}

\begin{proof}
This is a consequence of the Sebastiani-Thom isomorphism (see the results of N\'emethi \cite{nemethisebthom1},\cite{nemethisebthom2}, Oka \cite{okasebthom}, Sakamoto \cite{sakamoto}, Sebastiani-Thom \cite{sebthom}, and Massey \cite{masseysebthom}) for the reduced integral cohomology of the Milnor fiber of $f = g\boxplus h$ at $\0$. Letting  $\widehat C$ denote the component of the critical locus $f$ which corresponds to $C$, the Sebastiani-Thom Theorem tells us that 
$$\wt b_{n+m+1} (F_{f,\0})= \mu_\0(h)\wt b_{n} (F_{g,\0}), \hskip 0.2in \wt b_{n+m} (F_{f,\0})=  \mu_\0(h)\wt b_{n-1} (F_{g,\0}), \hskip 0.1in \textnormal{ and }\hskip 0.2in\mu^\circ_{{}_{\widehat C}}= \mu_\0(h) \mu^\circ_{{}_C}.$$

Thus,
\begin{align*}
\beta_f = \lm_{f,z_0}^0 - \lm_{f,z_0}^1 + \sum_{\widehat C} \mu^\circ_{{}_{\widehat C}} = \wt b_{n+m+1} (F_{f,\0}) - \wt b_{n+m} (F_{f,\0}) + \sum_{\widehat C} \mu^\circ_{{}_{\widehat C}} = \mu_\0(h)\beta_g.
\end{align*}
\end{proof}

\begin{cor}\label[corollary]{betainduction}
Suppose $f = g\boxplus h$, where $g$ and $h$ are as in \cref{gensusp}.  Then, if the \nameref{beta0} is true for $g$, it is true for $f$.  
\end{cor}

\begin{proof}
Suppose that $\beta_f = 0$.  By \cref{gensusp}, this is equivalent to $ \beta_g = 0$, since $\mu_\0(h) > 0$.  By assumption, $\beta_g =0$ implies that $\Sigma g$ is smooth at zero. Since $\Sigma f = \Sigma g \times \{ \0 \}$, it follows that $\Sigma f$ is also smooth at $\0$, i.e., the \nameref{beta0} is true for $f$.  
\end{proof}

\section{$\Gamma_{f,z_0}^1$ as a hypersurface in $\Gamma_{f,\bb{z}}^2$}\label{hypersurface}

Let $I := \langle \frac{\pd f}{\pd z_2},\cdots, \frac{\pd f}{\pd z_n} \rangle \subseteq \cO_{\cU,\0}$, so that the relative polar surface of $f$ with respect to the coordinates $\bb{z}$ is (as a cycle at $\0$) given by $\Gamma_{f,\bb{z}}^2 = \left [V(I) \right ]$. 

For the remainder of this section, we will drop the brackets around cycles for convenience, and assume that everything is considered as a cycle unless otherwise specified. We remind the reader that we are assuming that $f_{|_{V(z_0)}}$ has an isolated critical point at the origin.

\medskip

\begin{prop}\label{prop:equiv} The following are equivalent:

\begin{enumerate}

\item $\dim_\0\left(\Gamma_{f,\bb{z}}^2\cap V(f)\cap V(z_0)\right)=0$.

\item For all irreducible components $C$ at the origin of the analytic set $\Gamma_{f,\bb{z}}^2\cap V(f)$, $C$ is purely 1-dimensional and properly intersected by $V(z_0)$ at the origin.

\smallskip

\item $\Gamma_{f,\bb{z}}^2$ is properly intersected by $V(z_0, z_1)$ at the origin.

\end{enumerate}

\smallskip

Furthermore, when these equivalent conditions hold 
$$\left ( \Gamma_{f,\bb{z}}^2 \cdot V(f)\cdot  V(z_0) \right )_\0  \ = \ \mu_\0 \left ( f_{|_{V(z_0)}} \right ) + \left ( \Gamma_{f,\bb{z}}^2 \cdot V(z_0,z_1) \right )_\0.$$

\end{prop}

\begin{proof} Clearly (1) and (2) are equivalent. We wish to show that (1) and (3) are equivalent. This follows from Tessier's trick applied to $f_{|_{V(z_0)}}$, but -- as it is crucial -- we shall quickly run through the argument.

Since $f_{|_{V(z_0)}}$ has an isolated critical point at the origin,
$$
\dim_\0 \left(\Gamma_{f,\bb{z}}^2\cap V\left(\frac{\partial f}{\partial z_1}\right)\cap V(z_0)\right) =0.
$$
Hence, $Z:= \Gamma_{f,\bb{z}}^2\cap  V(z_0)$ is purely 1-dimensional at the origin.

Let $Y$ be an irreducible component of $Z$ through the origin, and let $\alpha(t)$ be a parametrization of $Y$ such that $\alpha(0)=\0$. Let $z_1(t)$ denote the $z_1$ component of $\alpha(t)$. Then,
$$
\big(f(\alpha(t))\big)' \ = \ {\frac{\partial f}{\partial z_1}}_{{\big |}_{\alpha(t)}}\cdot z_1'(t).\eqno{(\dagger)}
$$
Since $\dim_\0\displaystyle Y\cap V\left(\frac{\partial f}{\partial z_1}\right)=0$, we conclude that $\big(f(\alpha(t))\big)' \equiv 0$ if and only if $z_1'(t)\equiv 0$, which tells us that $f(\alpha(t)) \equiv 0$ if and only if $z_1(t)\equiv 0$. Thus, $\dim_\0Y\cap V(f)=0$ if and only if $\dim_\0 Y\cap V(z_1)=0$, i.e., (1) and (3) are equivalent. The equality now follows at once by considering the $t$-multiplicity of both sides of $(\dagger)$.
\end{proof}

\medskip

\vbox{\begin{thm}\label[theorem]{onlythm}
Suppose that 
\begin{enumerate}

\item for all irreducible components $C$ at the origin of the analytic set \, $\Gamma_{f,\bb{z}}^2\cap V(f)$, $C$ is purely 1-dimensional, properly intersected by $V(z_0)$ at the origin, and $\left(C\cdot V(z_0)\right)_\0 = \operatorname{mult}_\0 C$, and

\smallskip

\item the cycle $\Gamma_{f,z_0}^1$ equals $\Gamma_{f,\bb{z}}^2\cdot V(h)$ for some $h\in  \cO_{\cU,\0}$ (in particular, the relative polar curve at the origin is non-empty).

\end{enumerate}

Then, 
$$\wt b_n(F_{f,\0}) - \wt b_{n-1}(F_{f,\0})\geq  \left ( \Gamma_{f,\bb{z}}^2 \cdot V(z_0,z_1) \right )_\0$$
and so
$$\beta_f \geq \left ( \Gamma_{f,\bb{z}}^2 \cdot V(z_0,z_1) \right )_\0 +  \sum_{ C} \mu^\circ_{{}_{ C}}.$$   In particular, the \nameref{beta0} is true for $f$.  
\end{thm}
}

\begin{proof}
By Proposition \ref{prop:equiv}, 
$$\left ( \Gamma_{f,\bb{z}}^2 \cdot V(f)\cdot  V(z_0) \right )_\0  \ = \ \mu_\0 \left ( f_{|_{V(z_0)}} \right ) + \left ( \Gamma_{f,\bb{z}}^2 \cdot V(z_0,z_1) \right )_\0.$$

By assumption, $\Gamma_{f,z_0}^1 = \Gamma_{f,\bb{z}}^2\cdot V(h)$, for some $h\in  \cO_{\cU,\0}$.  Then, via \cref{intnums} and the above paragraph, we have
\begin{align*}
\wt b_n(F_{f,\0}) - \wt b_{n-1}(F_{f,\0})&=\lm_{f,z_0}^0 - \lm_{f,z_0}^1 \\
&= \left ( \Gamma_{f,z_0}^1 \cdot V(f) \right )_\0 - \mu_\0 \left ( f_{|_{V(z_0)}} \right ) \\
&=\left[\left ( \Gamma_{f,\bb{z}}^2 \cdot V(h)\cdot  V(f) \right )_\0 - \left ( \Gamma_{f,\bb{z}}^2 \cdot V(f)\cdot  V(z_0) \right )_\0\right] + \left ( \Gamma_{f,\bb{z}}^2 \cdot V(z_0,z_1) \right )_\0.
\end{align*}

As $\left(C\cdot V(z_0)\right)_\0 = \operatorname{mult}_\0 C$ for all irreducible components $C$ of  $\Gamma_{f,\bb{z}}^2 \cap V(f)$, the bracketed quantity above is non-negative.  The conclusion follows.\end{proof}

\medskip

\begin{exm}\label[example]{notsmooth}
To illustrate the content of \cref{onlythm}, consider the following example. Let $f = (x^3 + y^2 + z^5)z$ on $\CC^3$, with coordinate ordering $(x,y,z)$.  Then, we have $\Sigma f = V(x^3 + y^2,z)$, and
\begin{align*}
\Gamma_{f,(x,y)}^2 &= V \left ( \frac{\pd f}{\pd z} \right ) = V(x^3 + y^2 + 6z^5),
\end{align*}
which we note has an isolated singularity at $\0$.  

\smallskip

Then, 
\begin{align*}
V \left ( \frac{\pd f}{\pd y}, \frac{\pd f}{\pd z} \right ) &= V(2yz,x^3 + y^2 + 6z^5 ) \\
&= V(y,x^3 + 6z^5) + V(z,x^3 + y^2)  
\end{align*}
so that $\Gamma_{f,x}^1 = V(y,x^3 + 6z^5)$, and $\Lambda_{f,x}^1$ consists of the single component $C = V(z,x^3+y^2)$ with $\stackrel{\circ}{\mu}_C = 1$. It is then immediate that 
\begin{align*}
\Gamma_{f,x}^1 = V(y) \cdot \Gamma_{f,(x,y)}^2,
\end{align*}
so that the second hypothesis of \cref{onlythm} is satisfied.  For the first hypothesis, we note that 
\begin{align*}
\Gamma_{f,(x,y)}^2 \cap V(f) &= V( x^3 + y^2 + 6z^5, (x^3 + y^2 + z^5)z ) \\ 
&= V(5z^5,x^3+y^2+z^5) \cup V(x^3+y^2,z) \\
&= V(x^3+y^2,z ) = C.
\end{align*}
Clearly, $C$ is purely 1-dimensional, and is properly intersected by $V(x)$ at $\0$. Finally, we see that
\begin{align*}
(C \cdot V(x) )_\0 &= V(x,z,x^3 +y^2)_\0 = 2 = \mult_\0 C,
\end{align*}
so the two hypotheses of \cref{onlythm} are satisfied.  

\smallskip

By \cref{BettiLe}, \cref{onlythm} guarantees that the following inequality holds:
\begin{align*}
\lm_{f,x}^0 - \lm_{f,x}^1 &\geq \left ( \Gamma_{f,(x,y)}^2 \cdot V(x,y) \right )_\0.
\end{align*}
Let us verify this inequality ourselves.  We have
\begin{align*}
\lm_{f,x}^0 &= \left ( \Gamma_{f,x}^1 \cdot V \left ( \frac{\pd f}{\pd x} \right ) \right )_\0 = V(y,x^3+6z^5, 3x^2z)_\0 \\
&= V(y,x^2,z^5)_\0 + V(y,z,x^3)_\0 = 13,
\end{align*}
and 
\begin{align*}
\lm_{f,x}^1 &= \left ( \Lambda_{f,x}^1 \cdot V(x) \right )_\0 = V(x,z,x^3+y^2)_\0  = 2.
\end{align*}
Finally, we compute
\begin{align*}
\left ( \Gamma_{f,(x,y)}^2 \cdot V(x,y) \right )_\0 &= V(x,y,x^3 +y^2 + 6z^5)_\0 = 5.
\end{align*}
Putting this all together, we have 
\begin{align*}
\lm_{f,x}^0 - \lm_{f,x}^1 = 11 \geq 5 = \left ( \Gamma_{f,(x,y)}^2 \cdot V(x,y) \right )_\0,
\end{align*}
as expected.
\end{exm}

\begin{exm}\label[example]{nothypersurface}
We now give an example where the relative polar curve is {\bf not} defined inside $\Gamma_{f,\bb{z}}^2$ by a single equation, and $\wt b_n(F_{f,\0}) - \wt b_{n-1}(F_{f,\0}) < 0$.  

\smallskip

Let $f = (z^2 - x^2-y^2)(z-x)$, with coordinate ordering $(x,y,z)$.  Then, we have $\Sigma f = V(y,z-x)$, and 
\begin{align*}
\Gamma_{f,\bb{z}}^2 &= V \left ( \frac{\pd f}{\pd z} \right ) = V(2z(z-x) + (z^2-x^2-y^2) ).
\end{align*}
Similarly,
\begin{align*}
V \left ( \frac{\pd f}{\pd y}, \frac{\pd f}{\pd z} \right ) = V(y,3z+x) + 3 V(y,z-x),
\end{align*}
so that $\Gamma_{f,x}^1 = V(y,3z+x)$ and $\mu^\circ = 3$.  It then follows that $\Gamma_{f,x}^1$ is not defined by a single equation inside $\Gamma_{f,(x,y)}^2$ 

\smallskip

To see that $\wt b_2(F_{f,\0}) - \wt b_1(F_{f,\0}) < 0$, we note that, up to analytic isomorphism, $f$ is the homogeneous polynomial $f = (zx-y^2)z$.  Consequently, we need only consider the global Milnor fiber of $f$, i.e., $F_{f,\0}$ is diffeomorphic to $f^{-1}(1)$.  Thus, $F_{f,\0}$ is homotopy equivalent to $S^1$, so that $\wt b_2(F_{f,\0}) = 0$ and $\wt b_1(F_{f,\0}) = 1$.  

\end{exm}

\medskip

\begin{cor}\label[corollary]{nonredcor}
The \nameref{beta0} is true if the set $\Gamma_{f,\bb{z}}^2$ is smooth and transversely intersected by $V(z_0, z_1)$ at the origin.  In particular, the \nameref{beta0} is true for non-reduced plane curve singularities.
\end{cor}

\begin{proof} Suppose that the cycle $\Gamma_{f,\bb{z}}^2=m[V(\mathfrak p)]$, where $\mathfrak p$ is prime. Since the set $\Gamma_{f,\bb{z}}^2$ is smooth, $A:=\cO_{\cU,\0}/\mathfrak p$ is regular and so, in particular, is a UFD. The image of $\partial f/\partial z_1$ in $A$ factors (uniquely), yielding an $h$ as in hypothesis (2) of Theorem \ref{onlythm}.

Furthermore, the transversality of $V(z_0, z_1)$ to $\Gamma_{f,\bb{z}}^2$ at the origin assures us that, by replacing $z_0$ by a generic linear combination $az_0+bz_1$, we obtain hypothesis (1)  of Theorem \ref{onlythm}.
\end{proof}

\medskip

\begin{exm}\label[example]{SmoothEx}
Consider the case where $f = z^2 + (y^2-x^3)^2$ on $\CC^3$, with coordinate ordering $(x,y,z)$; a quick calculation shows that $\Sigma f = V(z,y^2-x^3)$.  Then, 
\begin{align*}
\Gamma_{f,(x,y)}^2 = V \left ( \frac{\pd f}{\pd z} \right ) = V(z)
\end{align*}
is clearly smooth at the origin and transversely intersected at $\0$ by the line $V(x,y)$, so the hypotheses of \cref{nonredcor} are satisfied.  Again, we want to verify by hand that the inequality
\begin{align*}
\lm_{f,x}^0 - \lm_{f,x}^1 \geq \left ( \Gamma_{f,(x,y)}^2 \cdot V(x,y) \right )_\0
\end{align*}
holds.  

\smallskip

First, we have
\begin{align*}
\lm_{f,x}^0 &= \left ( \Gamma_{f,x}^1 \cdot V \left ( \frac{\pd f}{\pd x} \right ) \right )_\0 = V(y,z,2(y^2-x^3)(-3x^2) )_\0 = V(y,z,x^5)_\0 = 5,
\end{align*}
and 
\begin{align*}
\lm_{f,x}^1 &= \left ( \Lambda_{f,x}^1 \cdot V(x) \right )_\0 
= V(x,z,y^2-x^3)_\0 = V(x,z,y^2)_\0 = 2.
\end{align*}
On the other hand, we have $\left ( \Gamma_{f,(x,y)}^2 \cdot V(x,y) \right )_\0 = V(x,y,z)_\0 = 1$, and we see again that the desired inequality holds.  
\end{exm}

\medskip

In the case where $f$ defines non-reduced plane curve singularity, there is a nice explicit formula for $\beta_f$, which we will derive in \cref{conclude}. 

\medskip

\section{Non-reduced Plane Curves}\label{conclude}

By \cref{nonredcor}, the \nameref{beta0} is true for non-reduced plane curve singularities.  However, in that special case, we may calculate $\beta_f$ explicitly.  

\smallskip

Let $\cU$ be an open neighborhood of the origin in $\CC^2$, with coordinates $(x,y)$. 

\begin{prop}\label[proposition]{nonred}
Suppose that $f$ is of the form $f = g(x,y)^p h(x,y)$, where $g : (\cU,\0) \to (\CC,\0)$ is irreducible, $g$ does not divide $h$, and $p > 1$.  Then,
 \begin{align*}
\beta_f &= \left \{ \begin{matrix} (p+1)V(g,h)_\0 + p\mu_\0(g) + \mu_\0(h) -1, && \text{ if $h(\0) = 0$; and} \\ p\mu_\0(g), && \text{ if $h(\0) \neq 0$.} \end{matrix} \right .
\end{align*}
Thus, $\beta_f = 0$ implies that $\Sigma f$ is smooth at $\0$.  

\end{prop}

\begin{proof}
After a possible linear change of coordinates, we may assume that the first coordinate $x$ satisifes $\dim_\bb{0} \Sigma (f_{|_{V(x)}}) = 0$, so that $\dim_\bb{0}V(g,x) = \dim_\bb{0}V(h,x) = 0$ as well. 

As germs of sets at $\bb{0}$, the critical locus of $f$ is simply $V(g)$.  As cycles, 
\begin{align*}
V \left ( \frac{\pd f}{\pd y} \right ) &= \Gamma_{f,x}^1 + \Lambda_{f,x}^1 =  V \left ( p h g^{p-1} \frac{\pd g}{\pd y} + g^p \frac{\pd h}{\pd y} \right )\\
 &= V \left ( ph\frac{\pd g}{\pd y} + g \frac{\pd h}{\pd y} \right ) + (p-1)V(g), 
\end{align*}

so that $\Gamma_{f,x}^1 = V \left ( ph\frac{\pd g}{\pd y} + g \frac{\pd h}{\pd y} \right ) $ and $\Sigma f$ consists of a single component $C = V(g)$. It is a quick exercise to show that, for $g$ irreducible, $g$ does not divide $\frac{\pd g}{\pd y}$, and so the nearby Milnor number is precisely $\mu^\circ_{{}_C} = (p-1)$ along $V(g)$.  

Suppose first that $h(\0) = 0$.  

Then, by \cref{intnums}, 
\begin{align*}
\lm_{f,x}^0 - \lm_{f,x}^1 &= \left ( \Gamma_{f,x}^1 \cdot V(f) \right )_\bb{0} - \mu_\bb{0} \left ( f_{|_{V(x)}} \right ).
\end{align*}

We then expand the terms on the right hand side, as follows:
\begin{align*}
\left ( \Gamma_{f,x}^1 \cdot V(f) \right )_\bb{0} &= p \left (\Gamma_{f,x}^1 \cdot V(g) \right )_\bb{0} + \left (\Gamma_{f,x}^1 \cdot V(h) \right )_\bb{0} \\
&= p V \left( g,h\frac{\pd g}{\pd y} \right )_\bb{0} + V \left (h,g \frac{\pd h}{\pd y} \right )_\bb{0} \\
&= (p+1)V(g,h)_\bb{0} + pV\left (g,\frac{\pd g}{\pd y} \right )_\bb{0} + V \left (h,\frac{\pd h}{\pd y} \right )_\bb{0}.
\end{align*}
Since $\dim_\bb{0} V(g,x) = 0$ and $\dim_\bb{0} V(h,x) = 0$, the relative polar curves of $g$ and $h$ with respect to $x$ are, respectively, $\Gamma_{g,x}^1 = V \left (\frac{\pd g}{\pd y} \right )$ and $\Gamma_{h,x}^1 = V \left (\frac{\pd h}{\pd y} \right )$.  We can therefore apply Teissier's trick to this last equality to obtain
\begin{align*}
\left ( \Gamma_{f,x}^1 \cdot V(f) \right )_\bb{0} &= (p+1)V(g,h)_\bb{0} + p \left [ V \left (\frac{\pd g}{\pd y},x \right )_\bb{0} + \mu_\bb{0}(g) \right ] + \left [ V\left (\frac{\pd h}{\pd y},x \right )_\bb{0} + \mu_\bb{0}(h) \right ] \\
&= (p+1)V(g,h)_\bb{0} + p\mu_\bb{0}(g) + pV(g,x)_\bb{0} + V(h,x)_\bb{0} - (p+1).
\end{align*}

Next, we calculate the Milnor number of the restriction of $f$ to $V(x)$:
\begin{align*}
\mu_\bb{0} \left ( f_{|_{V(x)}} \right ) &= V \left ( \frac{\pd f}{\pd y},x \right )_\bb{0} = \left (\Gamma_{f,x}^1 \cdot V(x) \right )_\bb{0} + (p-1)V(g,x)_\bb{0}.
\end{align*}
Substituting these equations back into our initial identity, we obtain the following:
\begin{align*}
\lm_{f,x}^0 - \lm_{f,x}^1 &= (p+1)V(g,h)_\bb{0} + V(g,x)_\bb{0} + V(h,x)_\bb{0} \\
&+ p\mu_\bb{0}(g) + \mu_\bb{0}(h) - \left (\Gamma_{f,x}^1 \cdot V(x) \right )_\bb{0} - (p+1).
\end{align*}

We now wish to show that $\left (\Gamma_{f,x}^1 \cdot V(x) \right )_\0 = V(gh,x)_\bb{0} -1$.  To see this, we first recall that 
\begin{align*}
\left ( \Gamma_{f,x}^1 \cdot V(x) \right )_\0 &= \mult_y \left \{  \left (p h \cdot \frac{\pd g}{\pd y} \right )_{|_{V(x)}} + \left (g \cdot  \frac{\pd h}{\pd y} \right )_{|_{V(x)}} \right \},
\end{align*}
where $g_{|_{V(x)}}$ and $h_{|_{V(x)}}$ are (convergent) power series in $y$ with constant coefficients.  If the lowest-degree terms in $y$ of $\left (p h  \frac{\pd g}{\pd y} \right )_{|_{V(x)}}$ and $\left (g  \frac{\pd h}{\pd y} \right )_{|_{V(x)}}$ do not cancel each other out, then the $y$-multiplicity of their sum is the minimum of their respective $y$-multiplicities, both of which equal $V(gh,x)_\0 -1$.   We must show that no such cancellation can occur.  To this end, let $g_{|_{V(x)}} = \sum_{i \geq n} a_i y^i$ and $h_{|_{V(x)}} = \sum_{i \geq m} b_i y^i$ be power series representations in $y$, where $n = \mult_y g_{|_{V(x)}}$ and $m = \mult_y h_{|_{V(x)}}$ (so that $a_n,b_m \neq 0$).  Then, a quick computation shows that the lowest-degree term of $\left (p h \frac{\pd g}{\pd y} \right )_{|_{V(x)}}$ is $pn \, a_n b_m$, and the lowest-degree term of $\left (g \frac{\pd h}{\pd y} \right )_{|_{V(x)}}$ is $m \, a_n b_m$.  Consequently, no cancellation occurs, and thus $\left (\Gamma_{f,x}^1 \cdot V(x) \right )_\0 = V(gh,x)_\0 -1 = n+m-1$.  

Therefore, we conclude that 
\begin{align*}
\beta_f &= (p+1)V(g,h)_\bb{0} + p\mu_\bb{0}(g) + \mu_\bb{0}(h) -1.
\end{align*}
Since $V(g)$ and $V(h)$ have a non-empty intersection at $\bb{0}$, the intersection number $V(g,h)_\bb{0}$ is greater than one (so that $\beta_f > 0$).

Suppose now that $h(\0) \neq 0$.  Then, from the above calculations, we find
\begin{align*}
\left ( \Gamma_{f,x}^1 \cdot V(f) \right )_\0 &= p \mu_\0(g) + p V(g,x)_\0 - (p+1), \text{ and } \\
\mu_\0 \left ( f_{|_{V(x)}} \right ) &=  p V(g,x)_\0 - 1
\end{align*}
so that $\beta_f = p \mu_\0(g)$. 

\smallskip

Recall that, as $\Sigma f = V(g)$, the critical locus of $f$ is smooth at $\bb{0}$ if and only if $V(g)$ is smooth at $\bb{0}$; equivalently, if and only if the Milnor number of $g$ at $\bb{0}$ vanishes.  Hence, when $\Sigma f$ is \emph{not} smooth at $\bb{0}$, $\mu_\bb{0}(g) > 0$, and we find that $\beta_f > 0$, as desired. 
\end{proof}

\medskip

\begin{rem}
Suppose that $f(x,y)$ is of the form $f = gh$, where $g$ and $h$ are relatively prime, and both have isolated critical points at the origin. Then, $f$ has an isolated critical point at $\bb{0}$ as well, and the same computation in \cref{nonred} (for $\mu_\0(f)$ instead of $\beta_f$) yields the formula
\begin{align*}
\mu_\bb{0}(f) = 2V(g,h)_\bb{0} + \mu_\bb{0}(g) + \mu_\bb{0}(h) -1.
\end{align*}
Thus, the formula for $\beta_f$ in the non-reduced case collapses to the ``expected value" of $\mu_\0(f)$  exactly when $p=1$ and $f$ has an isolated critical point at the origin.   
\end{rem}

\bigskip

\printbibliography

\end{document}